\documentclass[12pt]{amsart}

\usepackage{amsfonts,amsmath,latexsym,amssymb,verbatim,amsbsy}
\usepackage{amsthm}

\usepackage{pstricks}

\theoremstyle{plain}
\newtheorem{THEOREM}{Theorem}[section]

\newtheorem{theorem}[THEOREM]{Theorem}
\newtheorem{corollary}[THEOREM]{Corollary}
\newtheorem{lemma}[THEOREM]{Lemma}

\theoremstyle{definition}

\theoremstyle{remark}

\newcommand{\thm}[1]{Theorem~\ref{#1}}
\newcommand{\lem}[1]{Lemma~\ref{#1}}

\newcommand{\bb}{\begin{equation}}
\newcommand{\ee}{\end{equation}}
\newcommand{\bq}{\begin{eqnarray}}
\newcommand{\eq}{\end{eqnarray}}
\newcommand{\bqn}{\begin{eqnarray*}}
\newcommand{\eqn}{\end{eqnarray*}}

\newcommand{\N}{\ensuremath{\mathbb{N}}}   
\newcommand{\R}{\ensuremath{\mathbb{R}}}   

\renewcommand{\S}{\ensuremath{\mathbb{S}}}

\def \a {\alpha}
\def \b {\beta}
\def \d {\delta}

\def \e {\varepsilon}
\def \f {\varphi}

\def \l {\lambda}

\def \n {\nabla}
\def \s {\sigma}
\def \th {\theta}

\def \w {\omega}

\def \p {\partial}
\def \ra {\rightarrow}

\def\ocirc#1{\ifmmode\setbox0=\hbox{$#1$}\dimen0=\ht0 \advance\dimen0
  by1pt\rlap{\hbox to\wd0{\hss\raise\dimen0
  \hbox{\hskip.2em$\scriptscriptstyle\circ$}\hss}}#1\else {\accent"17 #1}\fi}

\DeclareMathOperator{\diver}{div} %
\DeclareMathOperator{\curl}{curl} %

\def \reg {\mathcal{R}}

\begin{document}

\title[Self-similar collapse]{On formation of a locally self-similar collapse in the incompressible Euler equations}
\author{Dongho Chae}
\author{Roman Shvydkoy}
\thanks{The work of D.C. was supported partially by NRF
grant 2006-0093854, while the work of R.S. was partially supported
by NSF grant DMS--0907812.}
\address[D. Chae]{Chung-Ang University, Department of Mathematics,
Dongjak-gu Heukseok-ro 84, Seoul 156-756, Republic of Korea}%
\email{dchae@cau.ac.kr}
\address[R. Shvydkoy]{University of Illinois at Chicago, Department of Mathematics
(M/C 249), Chicago, IL 60607, USA} %
\email{shvydkoy@math.uic.edu}

\subjclass[2000]{Primary:76B03; Secondary:35Q31}

\begin{abstract}
The paper addresses the question of existence of a locally self-similar blow-up for the incompressible Euler equations. Several exclusion results are proved based on the $L^p$-condition for velocity or vorticity and for a range of scaling exponents. In particular, in $N$ dimensions if in self-similar variables $u \in L^p$  and $u \sim \frac{1}{t^{\a/(1+\a)}}$, then the blow-up does not occur provided $\a >N/2$ or $-1<\a\leq N/p$. This includes the $L^3$ case natural for the Navier-Stokes equations. For $\a = N/2$ we exclude profiles with an asymptotic power bounds of the form
$ |y|^{-N-1+\d} \lesssim |u(y)| \lesssim |y|^{1-\d}$. Homogeneous near infinity solutions are eliminated as well except when homogeneity is scaling invariant. 
\end{abstract}

\maketitle

\section{Introduction}

In the theory of weak solutions to the Navier-Stokes equation one of the cornerstone results is non-existence of self-similar blow-up for velocities in $L^3$ proved by Ne{\v{c}}as, R{\ocirc{u}}{\v{z}}i{\v{c}}ka, and {\v{S}}ver{\'a}k, \cite{nrs}, and further extended to the case of $L^p$, $p>3$, by
Tsai \cite{tsai}. This was followed by the celebrated
$L^{3,\infty}$-regularity criterion of Escauriaza, Seregin, and
\v{S}ver{\'a}k \cite{ess}. For its inviscid counterpart, the Euler equation, given by
\begin{equation}\label{ee}
\begin{split}
u_t + u \cdot \n u + \n p & =0\\
\n \cdot u & = 0,
\end{split}
\end{equation}
the self-similar blow-up has not yet been explored systematically in mathematical literature despite abundance of numerical data based on \eqref{ee} pointing to such possibility. Brachet et al \cite{bmvps} observe a pancake-like formation of vortex structures from Taylor-Green initial condition. Simulations of Kerr \cite{kerr} present
strong evidence of a singularity corresponding to scaling $u \sim \frac{1}{\sqrt{T-t}}$, the same as for the Navier-Stokes.
Boratav and Pelz \cite{bp} tests on Kida's high-symmetry flows
reveal self-similar evolution of a focusing vortex dodecapole, again in the same scaling. Similar collapse was further observed in vortex filament models of Pelz \cite{pelz}, Kimura
\cite{kimura}, Ng and Bhattacharjee \cite{ngb}, and others.
 
To describe the mathematical setup, let us assume that the fluid domain is $\R^N$, although other choices are possible. Suppose that near some point $x^* \in \R^N$ a solution, initially starting from a smooth data, organizes into a locally self-similar blowup. In other words, there is a $\rho_0 >0$ and time $T>0$ such that
\begin{equation}\label{ansatz}
\begin{split}
u(x,t) & = \frac{1}{(T - t)^{\frac{\a}{1+\a}}} v\left( \frac{x-x^*}{(T - t)^{\frac{1}{1+\a}}}\right) \\
p(x,t) & =  \frac{1}{(T - t)^{\frac{2\a}{1+\a}}} q\left( \frac{x-x^*}{(T - t)^{\frac{1}{1+\a}}}\right),
\end{split}
\end{equation}
for all $|x-x^*|<\rho_0$, and $t<T$ near $T$, and where $\a>-1$ to
insure focusing collapse. Observe that the vorticity near
singularity scales like $\w = \curl v \sim \frac{1}{T-t}$, making it a borderline case
for the Beal-Kato-Majda criterion \cite{bkm}. The Lipschitz constant of the vorticity direction field $\xi = \frac{\w}{|\w|}$ scales like $(T-t)^{-\frac{1}{1+\a}}$, again in no
contradiction with Constantin and Fefferman's criterion \cite{c,cf}.
In \cite{he-bdd,he-ext} Xinyu He shows existence of solutions to
self-similar equations \eqref{e:ss} on bounded and exterior domains
with $\a = 1$. On exterior domains solutions exhibit the power-like
decay similar to vortex models, $|v| \sim |y|^{-1}$, $|\n v| \sim
|y|^{-2}$ under the same scaling. Although these solutions belong to different settings, interestingly, their decay rate appears critical for our results below.  One can observe  that $\a = N/2$ is the only scaling consistent with the energy conservation for globally self-similar solutions if the helicity is not zero (\cite{cha0}, see also \cite{schonbek} for `pseudo self-similar solutions'). A study of self-similar blow-up in the settings adopted here was undertaken by the first
author in a series of works \cite{cha,chae-11,chae-asym,chae-note}. The main two results obtained are the
following. First, if $v\in L^p(\R^3)$, $p\geq \frac{9}{2}$, and $\a
= \infty$, then  $v = 0$. Second, if $\|\n v \|_\infty <\infty$ and
the vorticity belongs to $\cap_{0<p<p_0}L^p(\R^3)$, for some $p_0$,
while $\a > - 1$  is arbitrary, then $v$ is irrotational, $\w = 0$
throughout.

In this paper we develop a new set of criteria that exclude
locally self-similar collapse in physically relevant scalings.
Let us observe that if the total energy of $u$ is finite, then by
rescaling the energy in the ball $|x-x^*|\leq \rho_0$, we have the
bound
\begin{equation}\label{Nalpha}
\int_{|y| < L \rho_0} |v(y)|^2 \lesssim L^{N-2\a}, \text{ for all } L> L_0.
\end{equation}
Therefore, the case $\a> \frac{N}{2}$ is automatically excluded, while in the range $\a<\frac{N}{2}$ the energy of  $v$ may be unbounded. In all our results we avoid  using the assumption of finiteness of total energy keeping in mind, for instance, the 3D vortex filament models where the energy is naturally unbounded.  We therefore examine the full range of $\a > - 1 $ and integrability conditions $v \in L^p$ for a possible collapse. If $v\in L^p$, $p>2$, there are two special values of $\a$ to consider:  $\a = \frac{N}{p}$ for the fact that $\|u\|_p$ is conserved under the self-similar evolution on the open space, and $\a = N/2$ as the boundary between local energy inflation and deflation regimes (see \eqref{Nalpha}). We will see that the cases $-1 <\a\leq \frac{N}{p}$, $\frac{N}{p}< \a \leq \frac{N}{2}$, and $\a > \frac{N}{2}$ are in fact different in character, and we exclude solutions  under the following  conditions:
\begin{itemize}
\item[(i)] $v \in L^p\cap C^1_{loc}$, $p \geq 3$, and $-1 <\a\leq \frac{N}{p}$ or $\a > \frac{N}{2}$;
\item[(ii)] $v \in L^2 \cap C^1_{loc}$, $\a = \frac{N}{2}$, and for some $\d >0$ and $|y|$ large, one has
\begin{equation}\label{i:pow}
\frac{c}{|y|^{N+1 - \d}} \leq |v(y) | \leq C |y|^{1-\d}.
\end{equation}
\end{itemize}
The local $C^1$-condition is only needed for the local energy equality to hold, and is natural since we view $T$ as the first time of regularity loss. The local energy equality will be our starting point in most arguments, although somewhat unusually for a self-similar problem, we will employ the full time-dependent version of it to be able to make a non-self-similar choice for a test function. As a result the local energy equality takes the form
\begin{equation}\label{i:bound}
\frac{1}{L^{N -2\a}} \int_{|y| \leq  L} |v|^2dy \lesssim \frac{1}{l^{N -2\a}} \int_{|y| \leq  l} |v|^2 dy + \int_{l \leq |y| \leq L}  \frac{|v|^3 + |v||q|}{|y|^{N+1 - 2\a}} dy.
\end{equation}
As we remarked above in asymptotic character of terms in \eqref{i:bound} depends on the range of $\a$ considered. Nontheless, \eqref{i:bound} allows us to control the growth of the energy either by the $L^p$-norm of $v$ on the large scales in case (i) or through the use of power bounds on $v$ as in (ii). This gives an improved  bound on the trilinear integral in \eqref{i:bound} by interpolation. The general strategy will then be to bootstrap between the growth of $L^2$ and $L^3$  norms of $v$ over large balls $|y|<L$ via a repeated use of \eqref{i:bound} until eventually the energy over $|y|<L$ deplays a decay as $L\ra \infty$, implying $v = 0$. It is precisely for $\frac{N}{p}< \a \leq \frac{N}{2}$ when this algorithm fails to bootstrap. However, as a byproduct of the argument, we obtain
\begin{itemize}
\item[(iii)] if $v \in L^p\cap C^1_{loc}$, $p\geq 3$, and $\frac{N}{p}< \a \leq \frac{N}{2}$, then \eqref{Nalpha} holds.
\end{itemize}
So, the energy growth bound \eqref{Nalpha} is a natural internal feature of the blow-up, independent of the total energy assumption. In particular, if $v \in L^p$, $p\geq 3$, and $\a = \frac{N}{2}$, then automatically $v \in L^2$.

Coming back to the vortex models or He's solutions, notice that in those cases $v \in L^p$ for $p>3$ (even if only at infinity) while $\a = 1$. So, they appear to be critical for the scope of (i).

We exhibit several explicit homogeneous examples of solution pairs $(v,q)$, see \eqref{noq}, \eqref{e:natur}, \eqref{e:psf}, \eqref{e:point}, which although lacking sufficient local regularity to be fully qualified as counterexamples, serve as indicators that our arguments may be sharp.  In Theorem \ref{t:homo} we demonstrate  however that locally smooth homogenous at infinity solutions are trivial unless the homogeneity is consistent with the scaling, and even then the case $\a = N/2$ is excluded.
 
A criterion dimensionally equivalent to (i), but in terms of vorticity, is established using the self-similar equations in vorticity form, generalizing the results obtained by the first author. We have
\begin{itemize}
\item[(iv)] Suppose $\a>-1$, $\w \in L^p$, for some $0<p<\frac{N}{1+\a}$, and the strain tensor $|\p v+\p^\top v| = o(1)$ as $|y| \ra \infty$. Then $v$ is a constant vector.
\end{itemize}

\section{Technical preliminaries}

\subsection{Self-similar equations and pressure}\label{ss:ss}  If $(u,p)$ is a distributional solution to \eqref{ee}, then the pair $(v,q)$ satisfies
 \begin{equation}\label{e:ss}
\frac{1}{1+ \a} y \cdot \n v + \frac{\a}{1+ \a} v = v \cdot \n v + \n q,
\end{equation}
 and  the pressure necessarily satisfies the Poisson equation
\begin{equation}\label{poisson}
\Delta q = - \diver \diver (v \otimes v) = - \p_i \p_j(v_i v_j).
\end{equation}
If $v \in L^p$, $2<p<\infty$ (resp., $L^\infty$) and  $q \in L^{p/2}$ (resp., BMO), then there is only one solution to \eqref{poisson} given by
\begin{equation}\label{q}
q(y) = - \frac{|v|^2}{N} + P.V. \int_{\R^N} K_{ij}(y-z) v_i(z)v_j(z)
\, dz,
\end{equation}
where the kernel is given by
\[
K_{ij}(y) = \frac{N y_i y_j - \d_{i,j}|y|^2}{N \w_N |y|^{N+2}},
\]
and $\w_N = 2\pi^{N/2}(N\Gamma(N/2))^{-1}$ is the volume of the unit ball in $\R^N$. The pressure given by \eqref{q} is referred to as the associated pressure. Unless stated otherwise we will always assume that the pressure is associated, however not for every pair $(v,q)$ solving \eqref{e:ss}, $q$ is given by \eqref{q}. Indeed, let
\begin{equation}\label{noq}
v = \langle 1,0 \rangle, \quad q = \frac{\a}{1+\a} y_1.
\end{equation}
This is a self-similar solution for any $\a > -1 $. Clearly, \eqref{q} does not hold (see \cite{kuk} for the role of such examples in uniqueness of solutions of the Navier-Stokes equation).

The equation in self-similar coordinates \eqref{e:ss} has its own intrinsic scaling -- if $v$ is a solution to \eqref{e:ss}, then
\[
v_\l(y) = \l v(y/\l), \quad q_\l(y) = \l^2 q(y/\l)
\]
is also a solution to the same equation. This suggests that in fact there may exist a non-trivial example of a $1$-homogeneous solution. And indeed, in 2D such an example is provided by
\begin{equation}\label{e:natur}
v(y) = \langle y_1, - y_2 \rangle, \quad q(y) = - y_2^2.
\end{equation}
Another example is the following parallel shear flow
\begin{equation}\label{e:psf}
v(y) = \langle y_2^\a, 0 \rangle, \quad q(y) = \frac{2\a}{(1+\a)^2} y_2^{\a+1},
\end{equation}
which in the case $\a = 1$ specifies to the natural homogeneity. A singular example of a solution of special interest to us is the $\a$-point vortex
\begin{equation}\label{e:point}
v(y) = \frac{y^\perp}{|y|^{\a + 1}}, \quad q(y) = 0.
\end{equation}

The equation for vorticity tensor $\w = \frac{1}{2} \{ \p_i v_j - \p_j v_i \}_{i,j = 1}^N$ in self-similar variables reads
\begin{equation}\label{vor}
\w + \frac{1}{1+\a} y \cdot \n \w = v \cdot \n \w -  \w \varsigma -
\varsigma \w,
\end{equation}
where $\varsigma = \frac{1}{2} \{\p_i v_j + \p_j v_i\}_{i,j = 1}^N $ is the strain tensor.

\subsection{Local energy equality} 
 All our results below hold under the presumption that the
 solution $(u,p)$ is regular enough to satisfy the local energy equality, at least in the region of self-similarity:
\begin{equation}\label{local-en}
\begin{split}
\int_{\R^N} |u(t_2,x)|^2 \f(t_2,x) \, dx - \int_{\R^N} |u(t_1,x)|^2 \f(t_1,x) \, dx = \\
\int_{t_1}^{t_2} \int_{\R^N} |u(t,x)|^2 \f_t(t,x) \, dx \,dt +   \int_{t_1}^{t_2}  \int_{\R^N} ( |u|^2 + 2p) u \cdot \n \f\, dx\, dt,
\end{split}
\end{equation}
where $\f \in C^\infty_0((0,T)\times \R^N)$, and $0<t_1<t_2<T$. It holds trivially for locally smooth finite energy solutions solutions, $u \in C^1_{loc}((0,T)\times \R^N) \cap L^\infty_t L^2_x$. The weakest condition under which \eqref{local-en} is known to hold is a Besov-type regularity of smoothness $1/3$ (see \cite{ccfs,shv-org}). It is not our goal however to pursure the sharpest local condition.

We will only be concerned with the local energy equality on the region of self-similarity. So, let us assume for simplicity and without loss of generality that $x^* = 0$, $\rho_0 = 1$, $T = 0$, while $t>0$. Let us fix a radial test function $\s$, i.e. $\s(x) = \s(|x|)$, such that $\s \geq 0$, $\s(r) = 1$, for $0 \leq r \leq \frac{1}{2}$, and $\s(r) = 0$, for $r >1$. Using $\f = \s$, \eqref{local-en} takes the form
\begin{equation}\label{eneqal}
\|u(t_2) \s \|_2^2  =   \|u(t_1) \s \|_2^2 + \int_{t_1}^{t_2}  \int_{\R^3} ( |u|^2 + 2p) u \cdot \n \s(x) dx dt.
\end{equation}
In self-similar variables the above translates into the following
\begin{equation}\label{eeq1}
\begin{split}
&t_2^{\frac{N-2\a}{1+\a}} \int_{|y| \leq t_2^{-\frac{1}{1+\a}}} |v(y)|^2 \s(y  t_2^{\frac{1}{1+\a}})\, dy = \\
&
t_1^{\frac{N-2\a}{1+\a}} \int_{|y| \leq t_1^{-\frac{1}{1+\a}}} |v(y)|^2 \s(y t_1^{\frac{1}{1+\a}})\, dy\, + \\
&\int_{t_1}^{t_2}  t^{\frac{N-3\a}{1+\a}} \int_{\frac{1}{2}t^{-\frac{1}{1+\a}} \leq |y| \leq t^{-\frac{1}{1+\a}}} (|v|^2 + 2q) v \cdot \n \s(y t^{\frac{1}{1+\a}})\, dy\, dt.
\end{split}
\end{equation}
Changing the order of integration in the last integral and changing notation in the first two with $l_i = t_i^{-1/(1+\a)}$, we obtain the following inequality for all $0<l_1<l_2$,
\begin{equation}\label{en-start}
\begin{split}
&\left| \frac{1}{l_2^{N- 2\a}}  \int_{|y| \leq l_2  } |v(y)|^2 \s(y/l_2) \, dy - 
\frac{1}{l_1^{N- 2\a}}  \int_{|y| \leq l_1 } |v(y)|^2 \s(y/l_1) \, dy \right| \\
& \leq C \int_{l_1/2 \leq |y| \leq l_2 } \frac{|v|^3 + |q||v|}{ |y|^{N+1 - 2\a }}\, dy.
\end{split}
\end{equation}
This inequality will be our starting point in much of what follows.

 \subsection{Global energy equality}
The global energy equality holds under additional $L^3$-integrability condition at infinity.
\begin{theorem} Let $u \in  C^w_t L^2_x \cap L^3_tL^3_x   \cap C^1_{loc}$ be a weak solution to the Euler equations on $\R^N$. Then $u$ conserves energy on $[0,T]$.
\end{theorem}
\begin{proof}  Let $\s_R(x) = \s(x/R)$. By the local energy equality we have
\[
\|u(t_2) \s_R \|_2^2 -  \|u(t_1) \s_R \|_2^2 = \int_{t_1}^{t_2} \frac{1}{R} \int_{\R^N} ( |u|^2 + 2p) u \cdot \n \s(x/R) dx dt.
\]
Since $u\in L^3_{t,x}$, then $p\in L^{3/2}_{t,x}$ and hence $( |u|^2 + 2p) u \in L^{1}_{t,x}$. Then, clearly, the integral on the right hand side tends to zero as $R \ra \infty$.
\end{proof}

As an immediate consequence we can eliminate certain self-similar solutions under the global energy law.
\begin{corollary} Suppose $u \in  C^w_t L^2_x \cap L^3_tL^3_x   \cap C^1_{loc}$ is a weak solution to the Euler equations on $\R^N$ with a locally self-similar collapse. If $\a >\frac{N}{2}$ then the collapse does not occur. Otherwise, \eqref{Nalpha} holds.
\end{corollary}

As a by-product of our proofs below we show that the conclusions of this corollary hold under only $L^p$-integrability assumption on the self-similar profile $v$. In other words, a self-similar solution even if viewed independently  from the ambient flow still behaves as if it was embedded in a global in space finite energy solution.

\section{Exclusions based on velocity}

\subsection{The energy conservative scaling $\a = \frac{N}{2}$}
As outlined in the introduction, the case of $\a = \frac{N}{2}$ is special since it is the only scaling compatible with the energy conservation law if \eqref{ansatz} was  defined globally in space. What distinguishes it from a pure technical point of view is the absence of weights in front of energy integrals in  the energy balance relation \eqref{en-start}. Our main result for this case is the exclusion of solutions with a power spread.

\begin{theorem} Let $v \in L^2(\R^N)\cap C^1_{loc}$ and the pressure $q$ given by \eqref{q}. Suppose there exists a $\d >0$ and $C,c >0$ such that
\begin{equation}\label{powers}
\frac{c}{|y|^{N+1 - \d}} \leq |v(y) | \leq C |y|^{1-\d},
\end{equation}
for all sufficiently large $y$. Then $v= 0 $.
\end{theorem}

A few comments are in order. Example \eqref{e:natur} shows relevance of the upper bound to the natural scaling of the equations,
although of course it has infinite energy. 
The lower bound may seem to be artificial especially given \thm{t:homo} below where homogeneous profiles with decay $|v| \sim |y|^{-\b}$ are excluded for any $\b \geq N/2$. However as we will see from the proof it is essentially a way of dealing with the non-locality of the pressure.

\begin{proof}

We start with the basic energy equality \eqref{en-start}. Using that $\a = \frac{N}{2}$, the factors in front of the energies disappear and we obtain
\begin{equation}\label{en-bound}
 \int_{|y| \leq l_2/2} |v|^2 \, dy \leq \int_{|y| \leq l_1 } |v|^2 \, dy + C \int_{l_1/2 \leq |y| \leq  l_2 } \frac{|v|^3 + |q||v|}{|y|}\, dy.
\end{equation}
Taking $l_1 = L = l_2/4$, we obtain
\begin{equation}\label{en-L}
 \int_{L \leq |y| \leq 2L} |v|^2 \, dy \leq C \int_{\frac{1}{2} L \leq |y| \leq 4  L} \frac{|v|^3 + |q||v|}{|y|} \, dy.
\end{equation}

The proof will now proceed by showing the following claim:  for all $M\in \N$ there exists a $C_M>0$ such that
\[
 \int_{L \leq |y| \leq 2L} |v|^2 dy \leq \frac{C_M}{L^M}.
 \]
 for all $L$ sufficiently large. This immediately runs into contradiction with the lower bound of \eqref{powers}. The exact value of the power $N+1 - \d$ is not important at this point, but it will be crucial in the course of proving the claim.

Using our assumption \eqref{powers} and the energy bound \eqref{en-L} we have
\begin{equation}\label{tech10}
\begin{split}
 \int_{L \leq |y| \leq 2L} |v|^2\, dy &\lesssim \frac{1}{L^{\d}} \int_{\frac{1}{2} L \leq |y| \leq 4  L} |v|^2\, dy \\ &+ \frac{1}{L} \int_{\frac{1}{2} L \leq |y| \leq 4  L} |v||q|\, dy.
\end{split}
\end{equation}
Now our goal is to find suitable bounds on the pressure and the last integral in \eqref{tech10}. Notice that
\begin{equation}\label{omega}
\int_{\S^{N-1}} K_{ij}(\th) d\s(\th) = 0,
\end{equation}
for all $i,j$. Let us split the pressure as follows
\[
q = q_0 + q_1 + q_2 + q_3,
\]
where $q_0$ is the local part of \eqref{q}, and
\[
\begin{split}
q_1(y) &= \int_{|z| \leq L/4} K_{ij}(y - z) v_i(z)v_j(z) \,dz,\\
q_2(y) &= \int_{L/4 \leq |z| \leq 8L} K_{ij}(y - z) v_i(z)v_j(z) \,dz,\\
q_3(y) &= \int_{|z| \geq 8 L} K_{ij}(y - z) v_i(z)v_j(z) \,dz.
\end{split}
\]
Clearly, only estimates on the non-local quantities $q_{i}$ are necessary. Since $|y - z| \sim L$ for all $|z| \leq L/4$ and  $\frac{1}{2} L \leq |y| \leq 4  L$, we have
\[
|q_1(y) | \lesssim \frac{1}{L^N} \int_{|z| \leq L/4} |v|^2 \,dz \leq  \frac{\|v\|_2^2}{L^N}.
\]
Thus, in view of \eqref{powers},
\[
\begin{split}
\frac{1}{L} \int_{\frac{1}{2} L \leq |y| \leq 4  L} |v||q_1|\,dy &\lesssim \frac{1}{L^{N+1}} \int_{\frac{1}{2} L \leq |y| \leq 4  L} |v|^2 |v|^{-1} \,dy \\
&\lesssim \frac{1}{L^\d} \int_{\frac{1}{2} L \leq |y| \leq 4  L} |v|^2\, dy.
\end{split}.
\]
As to $q_2$, we have
\[
\begin{split}
\frac{1}{L} \int_{\frac{1}{2} L \leq |y| \leq 4  L} |v||q_2|\,dy &\leq  \frac{1}{L} \left(\int_{\frac{1}{2} L \leq |y| \leq 4  L} |v|^2 dy \right)^{1/2} \left( \int_{\R^N} |q_2|^2 dy \right)^{1/2}\\
&\lesssim  \frac{1}{L} \left(\int_{\frac{1}{2} L \leq |y| \leq 4  L} |v|^2 dy \right)^{1/2} \left( \int_{L/4 \leq |y| \leq 8L} |v|^4 dy \right)^{1/2}\\
&\lesssim  \frac{1}{L^\d} \left(\int_{\frac{1}{2} L \leq |y| \leq 4  L} |v|^2 dy \right)^{1/2} \left( \int_{L/4 \leq |y| \leq 8L} |v|^2 dy \right)^{1/2}\\
&\lesssim  \frac{1}{L^\d}\int_{L/4 \leq |y| \leq 8L} |v|^2 dy .
\end{split}
\]
And as to $q_3$, we trivially have $|q_3(y)| \lesssim \frac{1}{L^N} \|v\|_2^2$. Thus,
\[
\begin{split}
\frac{1}{L} \int_{\frac{1}{2} L \leq |y| \leq 4  L} |v||q_3|\,dy & \lesssim \frac{1}{L^{N+1}} \int_{\frac{1}{2} L \leq |y| \leq 4  L} |v|\, dy \lesssim \frac{1}{L^\d} \int_{\frac{1}{2} L \leq |y| \leq 4  L} |v|^2 \, dy \\
&= 
\frac{1}{L^\d} \sum_{k=-1}^2 \int_{ 2^{k}L \leq |y| \leq 2^{k+1}L} |v|^2\, dy.
\end{split}
\]
Putting together the obtained estimates into \eqref{tech10} we conclude that there exists a constant $C>0$ such that for all $L$ large enough
\begin{equation}\label{tech11}
\begin{split}
 \int_{L \leq |y| \leq 2L} |v|^2\, dy \leq \frac{C}{L^\d} \sum_{k = -2}^3 \int_{2^kL \leq |y| \leq 2^{k+1}L} |v|^2\, dy.
\end{split}
\end{equation}
Let us now the iterate estimate above $m$ times applying it to each integral in the sum,
\[
\begin{split}
 \int_{L \leq |y| \leq 2L} |v|^2 dy &\leq \frac{C^m}{L^{m\d}} \sum_{k_1,\ldots,k_m = -2}^3  \int_{2^{k_1+\ldots+k_m}L \leq |y| \leq 2^{{k_1+\ldots+k_m}+1}L} |v|^2\, dy\\
 &\lesssim \frac{C_m}{L^{m\d}}.
\end{split}
\]
Since $m$ can be arbitrary, the claim is proved.
\end{proof}

\subsection{The energy non-conservative scaling $\a \neq \frac{N}{2}$}
As we mentioned earlier some cases of non-conservative scaling appear physically relevant. Additionally, in the range $-1<\a <\frac{N}{2}$, a possibly infinite energy of the self-similar profile $v$ is not in contradiction with the finiteness of the global energy as along as \eqref{Nalpha} holds. Our main result in the energy non-conservative scaling is the following.

\begin{theorem}\label{t:off} Suppose $v\in L^p \cap C^1_{loc}$ for some $3 \leq p \leq \infty$, and $q$ is given by \eqref{q}. If $-1< \a  \leq \frac{N}{p}$ or $ \frac{N}{2} < \a \leq  \infty$, then $v = 0$.
\end{theorem}
The scaling $\a = N/p$ is notable for the fact that the $L^p$-norm of the solution is conserved. If $\a <N/p$ it deflates as $t \ra 0$, and if $\a >N/p$ it inflates. The sharpness of this scaling is suggested by the $\a$-point vortex \eqref{e:point}. Even though it fails to satisfy the required regularity near the origin,  it does belong to $L^p$ near infinity precisely when $2/p < \a$.  He's solutions in exterior domains with asymptotic decay $|v(y)| \sim \frac{1}{|y|}$, hence in $L^3_{weak}$, are suggestive of criticality of $\a = N/p$ as well.

In the following we consider only the case when $p<\infty$ leaving the technicalities of the case $p=\infty$ to Section~\ref{sss:infty}.

\subsubsection{Proof in the range $-1< \a  \leq \frac{N}{p}$} In this range we can eliminate the $l_2$-integral from \eqref{en-start}. Our claim is
\[
\frac{1}{l_2^{N- 2\a}}  \int_{|y| \leq l_2  } |v(y)|^2 \s(y/l_2) \, dy \ra 0,
\]
as $l_2 \ra \infty$. Indeed, for a fixed large $M>0$ and $l_2>M$, we have
by the H\"older inequality,
\[
\begin{split}
 \frac{1}{l_2^{N- 2\a}}  \int_{|y| \leq l_2  } |v(y)|^2 \s(y/l_2) \, dy & \leq 
 \frac{1}{l_2^{N- 2\a}}  \int_{|y| \leq M  } |v(y)|^2 \, dy \\
 &+ l_2^{2\a - 2N/p} \left( \int_{M\leq |y| \leq l_2} |v|^p\, dy \right)^{2/p}.
\end{split}
\]
Letting $l_2 \ra 0$, the first integral disappears, and we have
\[
\leq \left( \int_{M\leq |y| } |v|^p\, dy \right)^{2/p} \ra 0,
\]
as $M \ra \infty$. So, \eqref{en-start} takes the form (using that $\s = 1 $ on $|y|<1/2$, and replacing $l_1/2 $ with $L$)
\begin{equation}\label{en-start2}
\begin{split}
\frac{1}{L^{N- 2\a}}  \int_{|y| \leq L } |v|^2  \, dy  \leq C \int_{L \leq |y| } \frac{|v|^3 + |q||v|}{ |y|^{N+1 - 2\a }}\, dy.
\end{split}
\end{equation}
By the H\"older inequality we obtain
\[
\frac{1}{L^{N- 2\a}}  \int_{|y| \leq L } |v|^2  \, dy  \leq C L^{2\a - 1 - 3N/p} \left( \int_{L \leq |y| } (|v|^3 + |q||v|)^{p/3}\, dy\right)^{3/p},
\]
and hence,
\begin{equation}\label{vbp}
 \int_{|y| \leq L } |v|^2  \, dy  \leq L^{\b_p}, \text{ where } \b_p = N-1 - \frac{3N}{p}.
\end{equation}
If $\b_p <0$, then the proof is finished by sending $L \ra \infty$. Otherwise,  by interpolation, we have
\begin{equation}\label{l3}
\int_{|y| \leq L} |v|^3 dy \leq CL^{\b_p\a_p}, \text{ where }\a_p = \frac{p-3}{p-2}.
\end{equation}
Coming back repeatedly to the inequality \eqref{en-start2} we will be able to bootstrap on the growth of energy based now on a better estimate for the $L^3$-norms \eqref{l3}. But first we have to establish the corresponding estimates on the growth of the pressure. 
\begin{lemma}\label{l:p} Let
\begin{equation}\label{l2l}
\int_{|y| \leq L} |v|^2 dy \leq CL^{a_2},
\end{equation}
and
\begin{equation}\label{l3l}
\int_{|y| \leq L} |v|^3 dy \leq CL^{a_3},
\end{equation}
hold for all large $L$, and $a_2 <N$, $\frac{3a_2 - N}{2} \leq a_3$. Then
\begin{equation}\label{p}
\int_{|y| \leq L} |q|^{3/2} dy \leq CL^{a_3}.
\end{equation}
\end{lemma}
In order not to verify the assumptions on the exponents every time, we simply note that they are verified for any couple $a_2$, $a_3$ with
\begin{equation}\label{a}
a_2 \leq N - \frac{2N}{p}, \quad a_3 = a_2 \a_p.
\end{equation}
Clearly, $a_2=\b_p$, $a_3=\b_p\a_p$ is such a couple.
\begin{proof}
Let, as before, $q = q_0 + \tilde{q}$, where $q_0$ is the local and $\tilde{q}$ is the non-local part of the pressure. We can split
\[
\begin{split}
\int_{|y| \leq L} |\tilde{q}|^{3/2} dy \leq \int_{|y| \leq L} \left| \int_{|z| \leq 2L}  K_{ij}(y-z) v_i(z) v_j(z) dz \right|^{3/2} dy + \\
+  \int_{|y| \leq L} \left| \int_{|z| \geq 2L} K_{ij}(y-z) v_i(z) v_j(z)  dz \right|^{3/2} dy = A+B.
\end{split}
\]
By the standard boundedness,
\[
A \leq C \int_{|z| \leq 2L} |v|^3 dz \leq C L^{a_3},
\]
as required. As to $B$, we use a dyadic decomposition,
\[
\begin{split}
B \leq \int_{|y| \leq L} \left(   \sum_{k=1}^\infty  \int_{ 2^k L \leq |z| \leq 2^{k+1}L} \frac{1}{|y-z|^N} |v(z)|^2 dz \right)^{3/2} dy.
\end{split}
\]
Given that $|y-z| \sim |z|$, we continue
\[
\begin{split}
B \leq L^N \left(   \sum_{k=1}^\infty   \frac{1}{2^{Nk} L^N} \int_{ 2^k L \leq |z| \leq 2^{k+1}L} |v(z)|^2 dz \right)^{3/2} \\
\leq \frac{C}{L^{N/2}} \left( \sum_{k=1}^\infty  \frac{2^{ka_2} L^{a_2}}{2^{Nk}} \right)^{3/2}
\leq C L^{\frac{3a_2 - N}{2}} \leq CL^{a_3},
\end{split}
\]
the latter holds due to imposed assumptions.
\end{proof}
Now using the obtained estimates \eqref{l3} and \eqref{p} in \eqref{en-start2} we obtain 
\[
\begin{split}
\frac{1}{L^{N- 2\a}}  \int_{|y| \leq L } |v|^2  \, dy  &\leq \frac{C}{L^{N+1-2\a}} \sum_{k=0}^\infty \frac{1}{2^{k(N+1-2\a)}} \int_{2^kL \leq |y| \leq 2^{k+1}L} (|v|^3 + |v||q|)\, dy\\
&\leq L^{\b_p\a_p - N - 1 +2\a} \sum_{k=0}^\infty 2^{k(\b_p\a_p - N - 1 +2\a)}.
\end{split}
\]
Notice that in the range $\a \leq N/p$ the power in the series is negative. Hence,
\begin{equation}\label{l22}
\int_{|y| \leq L} |v|^2 dy \leq CL^{\b_p \a_p - 1} \text{ and }\int_{|y| \leq L} |v|^3 dy \leq CL^{\b_p \a^2_p - \a_p}.
\end{equation}
Once again, the new exponents satisfy \eqref{a}, hence
\begin{equation}\label{pp}
\int_{|y| \leq L} |q|^{3/2} dy \leq CL^{\b_p \a^2_p - \a_p}.
\end{equation}
Substituting this into \eqref{en-start2} we obtain
\begin{equation}\label{l222}
\int_{|y| \leq L} |v|^2 dy \leq CL^{\b_p \a^2_p - \a_p -  1},
\end{equation}
and so on. Noting that on each step the assumptions on the exponents are satisfied (even improved), we arrive at
\begin{equation}\label{ln}
\int_{|y| \leq L} |v|^2 dy \leq CL^{\b_p \a^n_p - \a^{n-1}_p -... -   1}.
\end{equation}
For $n$ sufficiently large the power will become negative implying that $v = 0$.

\subsubsection{Proof in the range $\a >N/2$}
Starting from the same energy equality \eqref{en-start} we obtain
\[
\frac{1}{l_2^{N-2\a}} \int_{|y| \leq l_2/2} |v|^2 dy \lesssim \frac{1}{l_1^{N-2\a}} \int_{|y| \leq l_1 } |v|^2 dy +  \int_{l_1/2 \leq |y| \leq  l_2 } \frac{|v|^3 + |q||v|}{|y|^{N+1-2\a}}\, dy.
\]
Let us fix $l_1=2$ and $l_2 = 2 L >> 2$. Then
\begin{equation}\label{en10}
\int_{|y| \leq L} |v|^2\, dy  \lesssim L^{N-2\a} + L^{N-2\a} \int_{1 \leq |y| \leq  2L } \frac{|v|^3 + |q||v|}{|y|^{N+1-2\a}}\, dy,
\end{equation}
and by the H\"older,
\[
\lesssim L^{N-2\a} + L^{N-2\a} \left( \int_{1<|y|< 2L} \frac{1}{|y|^{(N+1-2\a)p/(p-3)}} \, dy \right)^{(p-3)/p}.
\]
Since $N-2\a<0$, the only case we have to consider is when $(N+1-2\a)p/(p-3) <N$. In this case the estimate above gives
\[
\int_{|y| \leq L} |v|^2 dy \lesssim L^{N-2\a} + L^{\b_p}.
\]
If $\b_p<0$ the proof is finished. Otherwise, we obtain
\begin{equation}\label{boot1}
\int_{|y| \leq L} |v|^2\, dy \lesssim L^{\b_p}, \text{ and } \int_{|y| \leq L} |v|^3\, dy \lesssim L^{\b_p\a_p}.
\end{equation}
We are in a position to intitiate the bootstrap argument as before, but with some  modifications. Pluging \eqref{boot1} in \eqref{en10} we find
\[
\begin{split}
\int_{|y| \leq L} |v|^2\, dy  & \lesssim L^{N-2\a} + \frac{1}{L} \sum_{k=-1}^{[\log_2 L]} 2^{k(N+1-2\a)} \int_{L/2^{k+1} < | y| < L/2^k } (|v|^3 + |q||v|)\, dy \\
&\lesssim L^{N-2\a} + L^{\b_p\a_p - 1} \sum_{k=-1}^{[\log_2 L]} 2^{k(N+1-2\a - \b_p\a_p)}. 
\end{split}
\]
If the power $N+1-2\a - \b_p\a_p \geq 0$, we obtain
\[
\lesssim L^{N-2\a} + L^{N-2\a} \log_2 L \ra 0, \text{ as } L \ra \infty.
\]
In this case the proof is over. Otherwise, we obtain
\[
\lesssim L^{N-2\a}  + L^{\b_p\a_p - 1}.
\]
If $\b_p\a_p - 1<0$, the proof is over. Otherwise, 
\[
\int_{|y| \leq L} |v|^2 \, dy \lesssim L^{\b_p \a_p - 1}, \text{ and }\int_{|y| \leq L} |v|^3 \, dy \lesssim L^{\b_p \a^2_p - \a_p}.
\]
The iteration will certainly terminate at a step when the  power 
\[
\b_p\a^n_p - \a_p^{n-1}-...-1
\] 
becomes negative, or earlier.

\subsubsection{Implications of the proof to the range $N/p < \a \leq N/2$}\label{sss:imply}
The proof above yields the following corollary.
\begin{corollary}\label{c:cons1} Suppose $ \frac{N}{p} < \a \leq \frac{N}{2}$. Then one has
\begin{equation}\label{goal}
\int_{|y| \leq L} |v|^2\, dy \lesssim L^{N - 2\a}.
\end{equation}
\end{corollary}
There is only one place of the argument which needs extra attention. That is if at some point we run into the logarithmic bound
\[
\int_{|y| \leq L} |v|^2\, dy \lesssim L^{N - 2\a} \log_2 L.
\]
Then for any $\e>0$ we have
\[
\int_{|y| \leq L} |v|^2\, dy \lesssim L^{N - 2\a+\e}, \text{ and }\int_{|y| \leq L} |v|^3\, dy \lesssim L^{(N - 2\a+\e)\a_p}.
\]
The conditions \eqref{a} are still satisfied for small $\e$, so the pressure has the analogous growth bound. Returning to \eqref{en10} and performing dyadic splitting of the integral as before we obtain
\[
\int_{|y| \leq L} |v|^2 dy   \lesssim L^{N-2\a} + L^{ (N-2\a + \e)\a_p - 1  }\sum_{k=0}^{[ \log_2 L ]} 2^{k( N+1-2\a - (N-2\a+\e)\a_p)}.
\]
The power in the sum is strictly positive. So, we obtain \eqref{goal}.

\subsubsection{\thm{t:off} in the case $p = \infty$}\label{sss:infty}
Only a few minor modifications are needed to extend the above argument to the case $v \in L^\infty$, $q \in BMO$. In the case $\a \leq 0$  we start from \eqref{eeq1} and subtract from $q$ the averages over dyadically divided time intervals. This, after changing the order as in \eqref{en-start} results in the following inequality
(in place of \eqref{en-start2})
\begin{equation}\label{en-start2}
\begin{split}
\frac{1}{L^{N- 2\a}}  \int_{|y| \leq L } |v|^2  \, dy  \leq C \sum_{k=1}^\infty \int_{2^k L \leq |y| \leq 2^{k+1} L } \frac{|v|^3 + |q - \bar{q}_k ||v|}{ |y|^{N+1 - 2\a }}\, dy,
\end{split}
\end{equation}
where $\bar{q}_k = \frac{1}{Vol(2^k L \leq |z| \leq 2^{k+1} L)} \int_{2^k L \leq |z| \leq 2^{k+1} L } q(z) dz$.
Using that 
\[
 \int_{2^k L \leq |y| \leq 2^{k+1} L } |q(y) - \bar{q}_k| dy \lesssim (2^k L)^N \|q\|_{BMO}
 \]
 we immediately obtain \eqref{vbp} with $\b_\infty = N-1$ as expected. Note that again the constants $\b_\infty$ and $\a_\infty = 1$ satisfy the requirements of \lem{l:p}. From this point on the argument proceeds as before.
 
 In the case $\a > 0$ a similar argument replaces \eqref{en10} with 
\[
\int_{|y| \leq L} |v|^2\, dy  \lesssim L^{N-2\a} + L^{N-2\a} \sum_{k=-1}^{[\log_2 L]} \int_{L/2^{k+1} < | y| < L/2^k }\frac{|v|^3 + |q - \bar{q}_k ||v|}{|y|^{N+1-2\a}}\, dy,
\]
where $\bar{q}_k = \frac{1}{Vol( L/2^{k+1} \leq |z| \leq  L/2^k )} \int_{ L/2^{k+1} \leq |z| \leq  L/2^k } q(z) dz$. The rest of the argument goes as before.

\section{Exclusions based on vorticity}\label{s:vor}

The condition in terms of vorticity that excludes a non-trivial blow-up stated and proved in \cite{cha} involves a requirement on decay at infinity in the sense that all $L^p$-norms for $0<p<p_0$ are finite. In this section we will eliminate solutions under a much weaker condition.
\begin{theorem}\label{t:vor}
 Suppose $v\in C^1_{loc}(\R^N)$ is a solution with $\a > -1$ satisfying the following conditions.
\begin{itemize}
\item[(i)] $| \varsigma(y) |=o(1)$ as $|y| \ra \infty$,
\item[(ii)] $\w \in L^p$, for some $0<p<\frac{N}{1+\a}$.
\end{itemize}
Then, $v$ is a constant vector field.
\end{theorem}
We note that He's examples \cite{he-ext}, although in different
settings, with $|\w| \sim \frac{1}{|y|^2}$ in 3D and $\a = 1$
correspond to $\w \in L^p$ for all $p>\frac{N}{1+\a}=\frac{3}{2}$. It points to the sharpness of our condition (ii). Furthermore, the value of $p =  \frac{N}{1+\a}$ appears naturally critical for the fact that the vorticity of the self-similar solution preserves this particular $L^p$-norm. Let us recall that for a similar reason the exponent $p^* = 
\frac{N}{\a}$ is critical for velocity in \thm{t:off} (i). The two are conjugate through the Sobolev embedding. Indeed, if $v \ra 0$ at infinity, $-1<\a\leq N-1$, then $\w \in L^p$ implies $v \in L^{p^*}$. This brings us back in agreement with the range of \thm{t:off} (i), although the end-point case cannot be excluded here.

\begin{proof} From (i) by the Fundamental Theorem of Calculus, the radial component of velocity is
\begin{equation}\label{ftc}
|v_r(y)|=o(|y|), \text{ as } |y| \ra \infty.
\end{equation}
Indeed, we have
\[
v(y) = v(0) + \int_0^1 \n v(ty) \cdot y \,dt.
\]
Then
\[
v_r(y) = v(y) \cdot \frac{y}{|y|} = v(0)\cdot  \frac{y}{|y|} + \frac{1}{|y|}\int_0^1 y \cdot \varsigma(ty)\cdot y \, dt,
\]
and the claim follows. Observe
$$\infty> \|\w\|_p^p =\int_{0} ^\infty \int_{|y|=r} |w |^p\, dS_r\, dr. $$
Hence, there
exists a sequence $R_j\uparrow \infty$ such that
$$R_j
\int_{|y|=R_j} |w |^p dS_{R_j} \to 0 \quad \mbox{as $j\to
\infty$}.
$$
 We multiply \eqref{vor} by
$\w|\w|^{p-2}$ and write it in the form
 \bb\label{vor1}
\begin{split}
 |\w|^p +\frac{1}{p(\a +1)}\mathrm{ div }\, (y|\w|^p) - \frac{N}{p(\a +1)}|\w|^p \\
 =\frac{1}{p} \mathrm{div} \,(v |\w|^p) -  \hat{\varsigma} |\w|^{p},
\end{split}
\ee
where $\hat{\varsigma} = (\w \varsigma \cdot \w + \varsigma \w \cdot \w)|\w|^{-2}$.
Let us fix an $R>0$, integrate \eqref{vor1} over the annulus $\{ R<|y|< R_j \}$, and apply the divergence theorem to have
\[
\begin{split}
 \left(\frac{N}{p(\a+1)} -1\right) \int_{R<|y|<R_j} |\w|^p dy
  +\frac{R}{p(\a+1)} \int_{|y|=R} |\w|^p dS_R \\
 - \frac{R_j}{p(\a+1)} \int_{|y|=R_j} |\w|^p
  dS_{R_j}\\
 =\int_{R<|y| <R_j} \hat{\varsigma} |\w|^p dy + \frac{1}{p}\int_{|y|=R} v_r|\w|^pdS_R  -\frac{1}{p}\int_{|y|=R_j} v_r |\w|^pdS_{R_j}.
\end{split}
\]
 Then, passing $j\to \infty$, one obtains
 \begin{align*}
 &\left(\frac{N}{p(\a+1)} -1\right) \int_{|y|>R} |\w|^p dy
  +\frac{R}{p(\a+1)} \int_{|y|=R} |\w|^p dS_R  \\
  &=\int_{|y| >R} \hat{\varsigma} |\w|^p dy+ \frac{1}{p}\int_{|y|=R} v_r|\w|^pdS_R
\intertext{and by choosing $R$ sufficiently large, and using (i) and \eqref{ftc}, we can ensure    }
&\leq \frac12\left( \frac{N}{p(\a+1)} -1\right) \int_{|y|>R} |\w|^p dy + \frac{R}{2p(\a+1)} \int_{|y|=R} |\w|^p dS_R.
\end{align*}
Consequently,
 $$
\int_{|y|>R} |\w|^p dy = \int_{|y|=R} |\w|^p dS_R=0,
$$
and hence, $\w=0$ on $\{ y\in \Bbb R^3\, |\,
  |y|>R\}$. Now we apply the result of \cite{cha} to conclude $\w=0$ on $\Bbb
  R^N$.
   Then there exists a harmonic function $h$ such that $v = \n h$. By (i), the Hessian matrix $\n \n h$ is bounded and vanishes at infinity. Since each entry is harmonic, by the Liouville Theorem, $\n\n h = 0$, and therefore $h$ is a quadratic polynomial. But, then from the condition $|\n h| = o(|y|)$, $\n h$ is constant.

  \end{proof}

\subsection{Homogeneous near infinity solutions} Given the plethora of 2D homogeneous examples in Section~\ref{ss:ss} it is natural to ask whether one can find a locally smooth self-similar profile homogenous near infinity. 
We say that a field $v \in C^1_{loc}(\R^N)$ is homogeneous near infinity if
\begin{equation}\label{homo}
v(y) = \frac{V(y |y|^{-1})}{|y|^\b}
\end{equation}
holds for all $y$ large enough and for some $V \in C^1(\S^{N-1}; \R^N)$.

\begin{theorem}\label{t:homo} Suppose $v$ is a homogeneous near infinity solution and any of these conditions are satisfied
\begin{itemize}
\item[(i)] $0<\b<\a$,
\item[(ii)] $-1<\a < \b$,
\item[(iii)] $ \a = \b = \frac{N}{2}$.
\end{itemize}
Then $v=0$, except in the case $\b =0$ which implies that $v$ is constant.
\end{theorem}

\begin{proof} 

In the case (i), since $\b >0$,  $v \in L^p$ for all $p > p_0$. If, in addition $\a >N/2$, then an application of \thm{t:off} concludes the proof. Otherwise, by Corollary \ref{c:cons1}, \eqref{goal} holds. On the other hand,
\[
\int_{L\leq |y| \leq 2 L} |v|^2\ dy \sim L^{N-2\b}  \int_{\S^{N-1}} |V(\th)|^2 dS(\th) 
\]
which necessitates $\b \geq \a$, unless $V = 0$. If $V = 0$, however, then \thm{t:vor} or the result of  \cite{cha} applies to find $v = \n h$ for some harmonic function $h$. Since $h =const$ near infinity, $h$ is constant throughout
by the Liouville Theorem, which implies $v = 0$.

In the case (ii) we have $|\n v| \sim \frac{1}{|y|^{\b+1}}$. Since $-1<\a < \b$, there exists a $p> 0$ with $\frac{N}{1+ \b} <p <\frac{N}{1+ \a}$. For this $p$, $\w \in L^p$, and \thm{t:vor} applies. Note that only in the case $\b = 0$ the constant velocity may be different from zero.

In the case (iii) Corollary \ref{c:cons1} implies $v \in L^2$. However for any $M>0$, 
\[
\int_{L\leq |y| \leq M L} |v|^2\ dy = \log M  \int_{\S^{N-1}} |V(\th)|^2 dS(\th).
\]
This implies $V = 0$ and the argument proceeds as before.

\end{proof}

Let us note that the main obstacle for extending \thm{t:homo} to the remaining case $-1 < \b \leq 0$ and $\b < \a$  is not inapplicability of Corollary~\ref{c:cons1}, but rather the lack of the corresponding bound
\begin{equation}\label{postq}
|q(y)| \lesssim |y|^{-2\b}
\end{equation}
for large $y$. Example \eqref{noq} demonstrates that \eqref{postq} may in fact fail for some solutions. But if bound \eqref{postq} is postulated then from the energy equality we obtain
\[
 \frac{1}{L^{N -2\a}} \int_{|y| \leq  L/2} |v|^2\, dy \leq\frac{1}{l^{N -2\a}} \int_{|y| \leq l} |v|^2\, dy + \int_{ l/2 \leq |y| \leq L}  \frac{|v|^3 + |v||q|}{|y|^{N+1 - 2\a}} \, dy.
\]
By a direct computation, with $l$ fixed, and $L$ large,
\[
\begin{split}
L^{N - 2\b} \|V\|_{L^2(\S^{N-1})}^2 \lesssim  \int_{|y| \leq  L/2} |v|^2\, dy &\lesssim L^{N-2\a} \\
&+ L^{N-2\a} \int_{c l \leq |y| \leq L} \frac{1}{|y|^{N+1 - 2\a+3\b}} \, dy.
\end{split}
\]
If $N+1 - 2\a+3\b \geq N$, then the above implies
\[
L^{N - 2\b}\|V\|_{L^2(\S^{N-1})}^2  \lesssim L^{N-2\a} \log L,
\]
and hence $V = 0$. Otherwise,
\[
L^{N - 2\b}\|V\|_{L^2(\S^{N-1})}^2 \lesssim  L^{N-2\a} + L^{N -1 -3\b},
\]
again, since $\b>-1$, implying $V=0$.


\end{document}